\documentclass[12pt, reqno]{amsart}

\title{On the density of $A+B$ when $B$ has few elements}

\usepackage[T1]{fontenc}
\usepackage{amsmath}
\usepackage{amssymb}
\usepackage{amsthm}
\usepackage[left=3.5cm, right=3.5cm, paperheight=11.8in]{geometry}
\usepackage{hyperref}
\usepackage{fancyhdr}
\usepackage{enumitem}
\usepackage{comment}
\usepackage{nicefrac}
\usepackage{mathrsfs}
\usepackage{bm}
\usepackage{graphicx}
\usepackage[utf8]{inputenc}
\usepackage{cancel}
\usepackage{mathtools}

\AtBeginDocument{%
   \def\MR#1{}
}

\newtheorem{thm}{Theorem}[section]
\newtheorem{cor}[thm]{Corollary}
\newtheorem{lem}[thm]{Lemma}
\newtheorem{prop}[thm]{Proposition}

\theoremstyle{definition} 
\newtheorem{defi}[thm]{Definition}
\let\olddefi\defi
\renewcommand{\defi}{\olddefi\normalfont}
\newtheorem{example}[thm]{Example}
\let\oldexample\example
\renewcommand{\example}{\oldexample\normalfont}
\newtheorem{rmk}[thm]{Remark}
\let\oldrmk\rmk
\renewcommand{\rmk}{\oldrmk\normalfont}

\author[P.~Leonetti]{Paolo Leonetti}
\address{
Universit\`{a} degli Studi dell'Insubria\\ via Monte Generoso 71 \\ Varese 21100\\ Italy}
\email{leonetti.paolo@gmail.com}


\hypersetup{
    pdftitle={On the density of $A+B$ when $B$ has few elements},
    pdfauthor={Paolo Leonetti},
    pdfmenubar=false,
    pdffitwindow=true,
    pdfstartview=FitH,
    colorlinks=true,
    linkcolor=blue,
    citecolor=green,
    urlcolor=cyan
}

\providecommand{\MR}[1]{}

\providecommand{\MR}{\relax\ifhmode\unskip\space\fi MR }

\begin{document}

\begin{abstract} 
Let $\mathsf{d}$ be the asymptotic density on the positive integers $\mathbb{N}^+$. We provide sufficient conditions on an infinite set $B\subseteq \mathbb{N}^+$ so that for each $\alpha \in [0,1]$ there exists a set $A\subseteq \mathbb{N}^+$ such that $\mathsf{d}(A)=0$ and $\mathsf{d}(A+B)=\alpha$. For instance, if $(b_n: n \in \mathbb{N}^+)$ is the increasing enumeration of $B$, it is sufficient that $\liminf_n b_{n+1}/b_n>1$. 
Moreover, we show that, if $|B\cap [1,n]|\ll \log(n)^{1-\varepsilon}$ for some $\varepsilon>0$, then for each $\alpha \in [0,1]$ there exists a set $A\subseteq \mathbb{N}^+$ such that $\mathsf{d}(A)=\mathsf{d}(A+B)=\alpha$. 
Conversely, we show that, for each $\varepsilon>0$, there exists an infinite set $B\subseteq \mathbb{N}^+$ such that $|B\cap [1,n]| \ll n^{2/3+\varepsilon}$ and $\mathsf{d}(A+B) \in \{0,1\}$ for every $A\subseteq \mathbb{N}^+$ such that $A+B$ admits asymptotic density. 
\end{abstract}
\maketitle
\thispagestyle{empty}

\section{Introduction and Main results}

Let $\mathsf{d}$ be the asymptotic density on the set of positive integers $\mathbb{N}^+:=\{1,2,\ldots\}$; more explicitly, if $S\subseteq \mathbb{N}^+$, then 
\begin{equation}\label{def:asympdensity}
\mathsf{d}(S):=\lim_{n\to \infty} \frac{|S\cap [1,n]|}{n},
\end{equation}
provided that the limit exists.  
Hereafter, given $S\subseteq \mathbb{N}^+$ and $\alpha \in [0,1]$, a notation of the type \textquotedblleft $\mathsf{d}(S)=\alpha$\textquotedblright\, stands for \textquotedblleft the limit in \eqref{def:asympdensity} exists and it is equal to $\alpha$.\textquotedblright

The sumset of two nonempty sets $A,B\subseteq \mathbb{N}^+$ is denoted by 
$$
A+B:=\{x+y: x\in A, y \in B\}.
$$
There is a long list of articles studying the relationships between the asymptotic density $\mathsf{d}$ and sumsets $A+B$, see e.g. \cite{MR4080054, MR4474049,  MR5517, MR64798, MR4197427, MR4009479,  MR4474053, MR4439850, MR4746441, MR63389, MR4879477, MR90618} and references therein. 

Here, we study a variant of the following result by Faisant et al., see \cite[Theorem 2.2]{MR4197427}, cf. also \cite[Theorem 3.2]{MR4439850} for a generalization. 
\begin{thm}\label{thm:grekos}
Let $B\subseteq \mathbb{N}^+$ be a nonempty finite set. Then for each $\alpha \in [0,1]$ there exists a set $A\subseteq \mathbb{N}^+$ such that $\mathsf{d}(A+B)=\alpha$.
\end{thm}

Following \cite[Definition 1.2]{MR4414919}, an infinite set $B\subseteq \mathbb{N}^+$ with increasing enumeration $(b_n: n \in \mathbb{N}^+)$ is said to be \emph{highly sparse} if 
$$
\lim_{n\to \infty} \frac{b_n}{b_{n+1}}=0.
$$
Chu claimed in \cite[Theorem 1.5]{MR4414919} that the analogue of Theorem \ref{thm:grekos} holds also for infinite sets $B\subseteq \mathbb{N}^+$ which are highly sparse. However, the same author realized that his proof is flawed \cite{MR4474049}, and asked whether his claim is actually correct in \cite[Question 0.4]{MR4474049}. 
We will provide an affirmative answer (see Corollary \ref{cor:chucorrected} below). 

In this regard, it is worth remarking that the analogue of Theorem \ref{thm:grekos} holds for 
all nonempty sets $B\subseteq \mathbb{N}^+$ which cover $o(k!)$ remainders modulo $k!$ as $k\to \infty$, that is, 
$B$ has Buck density $0$, see \cite[Theorem 2.3]{MR4746441}; cf. also \cite[Section 4]{LeoTri}. It is remarkable that many sets of integers satisfy the latter technical condition,   
including for instance the set of primes, the set of perfect powers, the set of factorials, etc., see \cite{MR4360486}. On the other hand, there exist highly sparse sets $B\subseteq \mathbb{N}^+$ which do not satisfy such condition, 
as for example $B=\{n!+n: n\in \mathbb{N}^+\}$, cf. \cite[Section 7]{LeoTri}. 

Our first main result follows: 
\begin{thm}\label{thm:chucorrectedtechnical}
    Let $B\subseteq \mathbb{N}^+$ be an infinite set with increasing enumeration $(b_n: n \in \mathbb{N}^+)$. Let us suppose that there exists a function $r: \mathbb{N}^+ \to \mathbb{N}^+$ for which the following hold\textup{:}
\begin{enumerate}[label={\rm (\roman{*})}]
        \item \label{item:0} $r(n)<n$ for all $n \ge 2$\textup{;}
        \item \label{item:1} $r(n)=o(n)$ as $n\to \infty$\textup{;}
        \item \label{item:2} $\mathsf{d}(E_r)=0$, where $E_r:=\bigcup_{k\ge 2} ([b_k,b_k+b_{k-r(k)}) \cap \mathbb{N}^+)$\textup{.}
\end{enumerate}
Then for each $\alpha \in [0,1]$ there exists a set $A\subseteq \mathbb{N}^+$ such that $\mathsf{d}(A)=0$ and $\mathsf{d}(A+B)=\alpha$.
\end{thm}

It is worth remarking, as we will show in Proposition \ref{prop:012-imply-stronger-growth} below, that if an infinite set $B\subseteq \mathbb{N}^+$ admits a function $r$ satisfying items \ref{item:0}--\ref{item:2} then necessarily $|B\cap [1,n]|=o_\varepsilon(n^\varepsilon)$ as $n\to \infty$.

One might conjecture that the existence of a function $r: \mathbb{N}^+ \to \mathbb{N}^+$ satisfying items \ref{item:0}--\ref{item:2} is so strong as to imply that $B$ has at least an exponential growth, namely, $|B\cap [1,n]|=O(\log n)$ as $n\to \infty$. However, this is false, as we show in the example below. 
\begin{example}\label{example:notlogn}
Define the infinite set $B=\{b_k: k \in \mathbb{N}^+\}$ and the map $r: \mathbb{N}^+\to \mathbb{N}^+$ by 
$$
\forall k \in \mathbb{N}^+, \quad 
b_k:=\lfloor e^{\sqrt{k}}\rfloor
\quad \text{ and }\quad 
r(k):=\lfloor k^{3/4}\rfloor
$$
Of course, we have $r(k)=o(k)$ as $k\to \infty$, and $r(k)<k$ for all $k\ge 2$. In addition, by the concavity of $x\mapsto \sqrt{x}$ and since $r(k) \ge \frac{1}{2}k^{3/4}$ as $k\to \infty$, we get 
$$
b_{k-r(k)} 
\le e^{\sqrt{k-r(k)}} 
\le e^{\sqrt{k}-\frac{r(k)}{2\sqrt{k}}} 
\le e^{\sqrt{k}-\frac{1}{4}\sqrt[4]{k}}
$$
for all sufficiently large $k \in \mathbb{N}^+$, which implies 
\begin{displaymath}
    \begin{split}
\sum_{k=2}^n b_{k-r(k)}
&\le O(1)+\sum_{2\le k\le n/2} e^{\sqrt{k}-\frac{1}{4}\sqrt[4]{k}}
+\sum_{n/2<k\le n} e^{\sqrt{k}-\frac{1}{4}\sqrt[4]{k}}\\
&\le O(1)+ ne^{\sqrt{n/2}}
+ne^{\sqrt{n}-\frac{1}{4}\sqrt[4]{n/2}}\\
&\le O(1)+o(e^{\sqrt{n}})+o(e^{\sqrt{n}})=o(e^{\sqrt{n}})=o(b_n)
    \end{split}
\end{displaymath}
as $n\to \infty$. Now, for each sufficiently large $N$, let $n$ be the unique integer such that $b_n\le N<b_{n+1}$. Since no interval defining $E_r$ with index larger than $n$ meets $[1,N]$, it follows that
\begin{displaymath}
        \frac{|E_r \cap [1,N]|}{N}
        \le \frac{\sum_{k=2}^n b_{k-r(k)}}{b_n}
        \longrightarrow 0
        \qquad\text{as }N\to\infty.
\end{displaymath}
Therefore $\mathsf{d}(E_r)=0$. Lastly, $|B\cap [1,b_n]|=n$ is not $O(\log(b_n))=O(\sqrt{n})$ as $n\to \infty$.
\end{example}

As a consequence of Theorem \ref{thm:chucorrectedtechnical}, 
we show that the analogue of Theorem \ref{thm:grekos} holds if the elements of the infinite set $B$ grow sufficiently fast:
\begin{thm}\label{thm:exponential}
Let $B\subseteq \mathbb{N}^+$ be an infinite set with increasing enumeration $(b_n: n \in \mathbb{N}^+)$ and suppose that 
$$
\exists q \in \mathbb{N}^+, \qquad 
\liminf_{n\to \infty}\frac{b_{n+q}}{b_n}>1.
$$
Then for each $\alpha \in [0,1]$ there exists a set $A\subseteq \mathbb{N}^+$ such that $\mathsf{d}(A)=0$ and $\mathsf{d}(A+B)=\alpha$.
\end{thm}

As a first application of Theorem \ref{thm:exponential}, we have the following (with $\mathbb{N}:=\{0,1,\ldots\}$): 
\begin{cor}\label{cor:sumwithfinite}
    Let $B\subseteq \mathbb{N}^+$ be an infinite set with increasing enumeration $(b_n: n \in \mathbb{N}^+)$ and suppose that 
$$
\liminf_{n\to \infty}\frac{b_{n+1}}{b_n}>1.
$$
Pick also a nonempty finite set $C\subseteq \mathbb{N}$. Then for each $\alpha \in [0,1]$ there exists a set $A\subseteq \mathbb{N}^+$ such that $\mathsf{d}(A)=0$ and $\mathsf{d}(A+B+C)=\alpha$.
\end{cor}

As another immediate consequence of Theorem \ref{thm:exponential}, since highly sparse sets $B$ satisfy $\lim_n b_{n+1}/b_n=\infty$ by definition, we conclude that Chu's open question \cite[Question 0.4]{MR4474049} has a positive answer 
(using Theorem \ref{thm:exponential} with $q=1$): 
\begin{cor}\label{cor:chucorrected}
Let $B\subseteq \mathbb{N}^+$ be an infinite set which is highly sparse. Then for each $\alpha \in [0,1]$ there exists a set $A\subseteq \mathbb{N}^+$ such that $\mathsf{d}(A)=0$ and $\mathsf{d}(A+B)=\alpha$.
\end{cor} 

As a last application of Theorem \ref{thm:exponential}, we include also cases such as $B=\bigcup_{n} \{2^n, 3^n\}$: 
\begin{cor}\label{cor:mixedfiniteunions}
    Let $B_1,\ldots,B_q\subseteq \mathbb{N}^+$ be nonempty sets such that $B:=B_1\cup \cdots \cup B_q$ is infinite. Suppose that each infinite $B_i$ has increasing enumeration $(b_{i,n}:n\in\mathbb N^+)$ satisfying $\liminf_n b_{i,n+1}/b_{i,n}>1$. 
%
Then for each $\alpha \in [0,1]$ there exists a set $A\subseteq \mathbb{N}^+$ such that $\mathsf{d}(A)=0$ and $\mathsf{d}(A+B)=\alpha$. 
\end{cor}

It is worth noting that Theorem \ref{thm:chucorrectedtechnical} is sufficiently flexible to include also cases of infinite sets $B$ with $\lim_n b_{n+1}/b_n=1$. 
\begin{thm}\label{thm:lastmainthm}
Let $B\subseteq \mathbb{N}^+$ be an infinite set with increasing enumeration $(b_n: n \in \mathbb{N}^+)$ and suppose that 
$$
\exists \lambda>1, \qquad 
\log\left(\frac{b_{n+1}}{b_n}\right)\gg \frac{(\log n)^\lambda}{n}.
$$
Then for each $\alpha \in [0,1]$ there exists a set $A\subseteq \mathbb{N}^+$ such that $\mathsf{d}(A)=0$ and $\mathsf{d}(A+B)=\alpha$.
\end{thm}
In fact, Theorem \ref{thm:lastmainthm} includes cases such as 
$B=\{\lfloor e^{n^\theta}\rfloor: n \in \mathbb{N}^+\}$ with $\theta \in (0,1)$, or 
$B=\{\lfloor e^{(\log n)^\beta}\rfloor: n \in \mathbb{N}^+\}$ with $\beta>2$.

On a similar direction, we provide an analogous result based only a sufficiently small growth of $B$ (which includes, in particular, $|B\cap [1,n]|\ll \log(n)^{1-\varepsilon}$ as $n\to \infty$ for some $\varepsilon>0$, as in the abstract). On the other extreme case, the requested set $A$ might have the same asymptotic density as $A+B$. 
\begin{thm}\label{thm:logarithmic}
Let $B\subseteq\mathbb{N}^+$ be an infinite set and suppose that 
$$
\limsup_{n\to\infty}
\frac{|B\cap[1,n]|\log\log n}{\log n}<1.
$$
Then for each $\alpha\in[0,1]$ there exists a set $A\subseteq\mathbb{N}^+$ such that
$\mathsf{d}(A)=\mathsf{d}(A+B)=\alpha$.
\end{thm}

Lastly, it is natural to ask whether the analogous results of Theorem \ref{thm:grekos} and Theorem \ref{thm:logarithmic} hold for \emph{all} sets $B\subseteq \mathbb{N}^+$ with $\mathsf{d}(B)=0$. Below, in our last main result, we provide a strong negative answer. 
\begin{thm}\label{thm:counterexample}
For each $\varepsilon>0$, there exists an infinite set $B\subseteq \mathbb{N}^+$ such that
$$
|B\cap [1,n]|=O(n^{2/3+\varepsilon})
\quad \text{ as }n\to \infty,
$$
and $\mathsf{d}(A+B)\in \{0,1\}$ for every set $A\subseteq \mathbb{N}^+$ such that $A+B$ admits asymptotic density.
\end{thm}

It is worth noting that 
Theorems \ref{thm:logarithmic} and \ref{thm:counterexample} answer \cite[Question 3]{MR4197427}: the former gives a sufficient condition on $B$, while the latter supplies a zero-density counterexample. Theorem \ref{thm:counterexample} also answers negatively the first parts of \cite[Questions 1.1 and 1.6]{MR4414919}, with equal prescribed lower and upper densities in the latter question. 
The proofs of our results follow in the next sections.

\section{Proof of Theorem \ref{thm:chucorrectedtechnical}}

Define 
$$
\forall n \in \mathbb{N}^+, \quad \quad 
p_n:=1-(1-\alpha)^{1/\kappa(n)},
$$
where 
$$
\kappa(n):=|B\cap [1,n]|.
$$ 
We use the convention that $1/0:=0$ and $0^0:=1$, so that $p_n=0$ for all $n<\min B$.  

We start by proving that the existence of the map $r$ satisfying items \ref{item:0}--\ref{item:2} implies that the elements of $B$ grow superpolynomially. 

\begin{prop}\label{prop:012-imply-stronger-growth}
For each $\varepsilon>0$, we have
$$
\kappa(N)=o_\varepsilon\!\left(N^\varepsilon\right)
\qquad\text{as }N\to\infty.
$$
In particular, $\mathsf{d}(B)=0$. 
\end{prop}

\begin{proof}
Set $N_k:=b_k+b_{k-r(k)}-1$ for all $k\ge 2$. Since $r$ satisfies items \ref{item:0}--\ref{item:2}, we obtain
$$
\frac{b_{k-r(k)}}{b_k+b_{k-r(k)}-1}
\le \frac{|E_r\cap[1,N_k]|}{N_k}=o(1)
\qquad\text{as }k\to\infty.
$$
This implies that $b_{k-r(k)}=o(b_k)$ as $k\to \infty$. Define
$$
\forall k\ge 2, \qquad g(k):=\log b_k,
$$
and notice that
$g(k)-g(k-r(k))=\log\!\left(\frac{b_k}{b_{k-r(k)}}\right)\to \infty$
as $k\to \infty$.

Fix an integer $M\ge 2$. By the previous observation and since $r(k)=o(k)$, there exists $k_M\in\mathbb N^+$ such that
\begin{equation}\label{eq:twoestimates}
\forall k\ge k_M, \qquad
r(k)\le \frac{k}{2M}
\quad\text{and}\quad
g(k)-g(k-r(k))\ge M.
\end{equation}

Now, fix an integer $n_0\ge 2k_M$, and define recursively
$$
n_{j+1}:=n_j-r(n_j)
$$
until the first index $t$ for which $n_t<n_0/2$. Since
$r(n_j)\le n_j/(2M)\le n_0/(2M)$ by \eqref{eq:twoestimates}, fewer than $M$ steps cannot decrease the value from $n_0$ to below $n_0/2$. Hence $t\ge M$. Moreover, $n_{t-1}\ge n_0/2$, and therefore, using \eqref{eq:twoestimates} at each step,
\begin{displaymath}
\begin{split}
g(n_0)-g(\lfloor n_0/2\rfloor)
&\ge g(n_0)-g(n_{t-1})\\
&=\sum_{j=0}^{t-2}\bigl(g(n_j)-g(n_{j+1})\bigr)\ge (t-1)M\ge M(M-1).
\end{split}
\end{displaymath}
Iterating this inequality along dyadic scales, we find that for every integer $J\ge 1$ such that $n_0/2^J\ge 2k_M$,
\begin{displaymath}
\begin{split}
g(n_0)
&\ge \sum_{j=0}^{J-1}\Bigl(g\!\left(\Bigl\lfloor\frac{n_0}{2^j}\Bigr\rfloor\right)
-g\!\left(\Bigl\lfloor\frac{n_0}{2^{j+1}}\Bigr\rfloor\right)\Bigr)\ge J\,M(M-1).
\end{split}
\end{displaymath}
Choosing
$J=\left\lfloor \log_2\!\left(n_0/(2k_M)\right)\right\rfloor$,
we conclude that
$$
g(n_0)\ge M(M-1)
\left\lfloor \log_2\!\left(n_0/(2k_M)\right)\right\rfloor
$$
for all sufficiently large $n_0$. Since the integer $M\ge2$ is arbitrary, this proves that
$$
\lim_{k\to \infty}\frac{g(k)}{\log k}=\infty.
$$

Now, fix $\varepsilon>0$. By the above limit, 
$$
\varepsilon g(n)-\log n\longrightarrow\infty
\qquad\text{as }n\to\infty,
$$
and therefore
$$
\frac{n}{b_n^\varepsilon}
=
\exp\bigl(\log n-\varepsilon g(n)\bigr)
\longrightarrow 0
\qquad\text{as }n\to\infty.
$$
Thus $n=o(b_n^\varepsilon)$ as $n\to\infty$. Finally, if $b_n\le N<b_{n+1}$, then $\kappa(N)=n$, and hence
$$
\frac{\kappa(N)}{N^\varepsilon}
=
\frac{n}{N^\varepsilon}
\le
\frac{n}{b_n^\varepsilon}
\longrightarrow 0
\qquad\text{as }N\to\infty.
$$
This concludes the proof.
\end{proof}

We split the rest of the proof into three cases.

\medskip 

\textbf{Case $\bm{\alpha=0}$.} It is enough to choose $A=\{1\}$, so that $\mathsf{d}(A)=0$ and, thanks to Proposition \ref{prop:012-imply-stronger-growth}, $\mathsf{d}(A+B)=\mathsf{d}(B)=0$. 

\medskip

\textbf{Case $\bm{\alpha=1}$.} Since $B\subseteq \mathbb{N}^+$ is an infinite set, it is a well-known result by Lorentz (which solved a conjecture by Erd\"os) that there exists $A\subseteq \mathbb{N}^+$ such that $\mathsf{d}(A)=0$ and $A+B$ is cofinite, see \cite{MR63389}. In particular, $\mathsf{d}(A+B)=1$. 

\medskip

\textbf{Case $\bm{\alpha\in (0,1)}$.}  
Pick a probability measure space $(\Omega, \mathscr{F}, P)$ and independent Bernoulli random variables $(Y_n: n \in \mathbb{N}^+)$ such that $P(Y_n=1)=p_n$ for all $n \in \mathbb{N}^+$. Then, for each $\omega \in \Omega$, define 
$$
A(\omega):=\{n \in \mathbb{N}^+: Y_n(\omega)=1\}.
$$

For each $m \in \mathbb{N}^+$, let also $X_m$ be the Bernoulli random variable defined by
\begin{displaymath}
    X_m(\omega):=
    \begin{cases}
        \,1\,\,\,\,& \text{ if }m \in A(\omega)+B;\\
        \,0\,\,\,\,& \text{ otherwise}.
    \end{cases}
\end{displaymath}
It follows that 
\begin{equation}\label{eq:meanY_m}
\mathbb{E}X_m=1-\prod_{j:\, b_j<m}(1-p_{m-b_j}).
\end{equation}
Indeed, the integers $m-b_j$ are pairwise distinct, and the event $(m \in A+B)$ holds exactly when at
least one of those numbers belongs to $A$.

For each integer $k\ge 2$, define 
$$
\alpha_k:=1-(1-\alpha)^{\frac{k}{k-r(k)}},
$$
and note that $\lim_k \alpha_k=\alpha$ by item \ref{item:1}. 

\begin{lem}\label{lem:basicestimate}
    Pick $m,k \in \mathbb{N}^+$ such that $k\ge 2$ and $b_k+b_{k-r(k)}\le m<b_{k+1}$. Then 
    $$
    \alpha \le \mathbb{E}X_m\le \alpha_k. 
    $$
\end{lem}
\begin{proof}
    Since $b_k<b_k+b_{k-r(k)}\le m<b_{k+1}$, we have  $\kappa(m)=k$. 
    In addition, for each $j=1,\ldots,k$, we obtain
    $
    b_{k-r(k)}\le m-b_k\le m-b_j<m<b_{k+1},
    $ 
    so that
    $$
    \forall j=1,\ldots,k, \quad 
    k-r(k)\le \kappa(m-b_j)\le k.
    $$
    This implies that 
    $
    \sum_{j=1}^k \frac{1}{\kappa(m-b_j)}\ge \sum_{j=1}^k \frac{1}{k}=1.
    $ 
    It follows by \eqref{eq:meanY_m} and the definition of the probabilities $p_n$ that
    \begin{displaymath}
            \mathbb{E}X_m
            =1-\prod_{j=1}^k (1-p_{m-b_j})
            =1-(1-\alpha)^{\sum_{j=1}^k \frac{1}{\kappa(m-b_j)}}\ge \alpha.
    \end{displaymath}
Similarly, since $\sum_{j=1}^k \frac{1}{\kappa(m-b_j)}\le \sum_{j=1}^k \frac{1}{k-r(k)}=\frac{k}{k-r(k)}$, we conclude $\mathbb{E}X_m \le \alpha_k$. 
\end{proof}

Now, observe that $E_r$ contains every integer $m$ in a block $[b_k,b_{k+1})$, with $k\ge2$, whose expectation $\mathbb{E}X_m$ is not covered by Lemma \ref{lem:basicestimate}; indeed,
$$
E_r=\bigcup_{k\ge 2}([b_k,b_k+b_{k-r(k)}) \cap \mathbb{N}^+).
$$
Recall by item \ref{item:2} that $\mathsf{d}(E_r)=0$. Define  
$$
\forall N \in \mathbb{N}^+, \quad S_N:=X_1+X_2+\cdots+X_N,
$$
so that $S_N(\omega)=|(A(\omega)+B) \cap [1,N]|$ for each $\omega \in \Omega$ and $N \in \mathbb{N}^+$. 

\begin{prop}\label{prop:ESN}
    $
    \mathbb{E}S_N=\alpha N+o(N)
    $ as $N\to \infty$. 
\end{prop}
\begin{proof}
    Fix $\varepsilon>0$. Since $\lim_k\alpha_k=\alpha$, there exists $k_0 \in \mathbb{N}^+$ such that $|\alpha_k-\alpha|<\varepsilon$ for all $k\ge k_0$. It follows by Lemma \ref{lem:basicestimate} that 
    $$
    \forall m \in (\mathbb{N}^+\setminus E_r)\cap [b_{k_0},\infty), \quad \quad 
    |\mathbb{E}X_m-\alpha|<\varepsilon.
    $$
    Taking into account that $X_m \in \{0,1\}$ for each $m \in \mathbb{N}^+$, this implies that 
    $$
    \forall N \in \mathbb{N}^+, \quad 
    |\mathbb{E}S_N-\alpha N| \le b_{k_0}+|E_r \cap [1,N]|+\varepsilon N.
    $$
    
    Since $\mathsf{d}(E_r)=0$, we obtain that $|\mathbb{E}S_N-\alpha N| \le 2 \varepsilon N$ for all sufficiently large $N \in \mathbb{N}^+$. The claim follows by the arbitrariness of $\varepsilon$. 
\end{proof}

\begin{lem}\label{lem:concentration}
    For all $\varepsilon>0$, there exists an integer $N_\varepsilon\ge 2$ such that 
    $$
    \forall N \ge N_\varepsilon, \quad 
    P\left(|S_N-\mathbb{E}S_N|>\varepsilon N\right)\le 
    \frac{1}{N^2}. 
    $$
\end{lem}
\begin{proof}
    Fix an integer $N\ge 2$. Observe that changing the value of a single $Y_m$ (i.e., whether $m \in A$ or not), for an index $m \in \{1,\ldots,N\}$, modifies the value of $S_N=|(A+B) \cap [1,N]|$ by at most $\kappa(N)=|B\cap [1,N]|$. Considering $S_N$ as a function of the random variables $Y_1,\ldots,Y_N$, it follows by McDiarmid's inequality \cite[Lemma 1.2]{MR1036755} that 
    $$
    P\left(|S_N-\mathbb{E}S_N|>\varepsilon N\right)\le 2\, \mathrm{exp}\left(-2 \frac{(\varepsilon N)^2}{N \cdot \kappa(N)^2}\right)
    =2\, \mathrm{exp}\left(-2\,\varepsilon^2 \frac{N}{\kappa(N)^2}\right).
    $$

    To conclude, it follows by Proposition \ref{prop:012-imply-stronger-growth} that $\kappa(N)=o(N^{1/4})$ as $N\to\infty$. Hence
$$
\frac{N}{\kappa^2(N)\log N}\longrightarrow\infty
\qquad\text{as }N\to\infty.
$$
Therefore, there exists an integer $N_\varepsilon\ge 2$ such that
$$
\frac{N}{\kappa^2(N)} \ge \frac{2}{\varepsilon^2} \log(N)
$$
for all $N\ge N_\varepsilon$. This implies that
$$
P\left(|S_N-\mathbb{E}S_N|>\varepsilon N\right)\le
2\, \mathrm{exp}\left(-4 \log(N)\right)=\frac{2}{N^4} \le \frac{1}{N^2}
$$
for all $N\ge N_\varepsilon$, which proves our claim.
\end{proof}

\begin{prop}\label{prop:concentration}
    We have $P$-almost surely that $S_N=\mathbb{E}S_N+o(N)$ as $N\to \infty$.
\end{prop}
\begin{proof}
Pick a decreasing sequence $(\varepsilon_k)_{k\ge 1}$ of positive reals with $\lim_k\varepsilon_k=0$. Now, fix $k \in \mathbb{N}^+$ and observe by Lemma \ref{lem:concentration} that 
$$
\sum_{N\ge 2} P\left(\frac{|S_N-\mathbb{E}S_N|}{N}>\varepsilon_k \right)\le N_{\varepsilon_k}+\sum_{N=2}^\infty \frac{1}{N^2}<\infty. 
$$
    It follows by 
    the first Borel-Cantelli lemma \cite[Theorem 4.18]{MR4226142} that 
    \begin{equation}\label{eq:varepsiloncondition}
    \limsup_{N\to \infty}\frac{|S_N-\mathbb{E}S_N|}{N}\le \varepsilon_k
    \quad P\text{-almost surely}. 
    \end{equation}

    Since Inequality \eqref{eq:varepsiloncondition} holds for all $k \in \mathbb{N}^+$, it follows that $\lim_N |S_N-\mathbb{E}S_N|/N=0$ $P$-almost surely.
\end{proof}

Putting together Proposition \ref{prop:ESN} and Proposition \ref{prop:concentration}, there exists an $\mathscr{F}$-measurable subset $\Omega_0 \subseteq \Omega$ such that $P(\Omega_0)=1$ and 
\begin{equation}\label{eq:finalclaim1}
\forall \omega \in \Omega_0, \quad 
\lim_{N\to \infty}\frac{|(A(\omega)+B) \cap [1,N]|}{N}
=\lim_{N\to \infty}\frac{S_N(\omega)}{N}
=\lim_{N\to \infty}\frac{\mathbb{E}S_N}{N}=\alpha. 
\end{equation}

Lastly, for each $N \in \mathbb{N}^+$ define the random variables $T_N: \Omega\to \mathbb{R}$ by 
$$
T_N:=Y_1+\cdots+Y_N,
$$
so that $T_N(\omega)=|A(\omega)\cap [1,N]|$ for all $\omega \in \Omega$. 
Since $B$ is infinite, we have $\kappa(n)\to\infty$, and therefore $\lim_n p_n=0$. This implies that $\mathbb{E}T_N=\sum_{n\le N}p_n=o(N)$ as $N\to \infty$. 

By Hoeffding's inequality \cite[Theorem 2]{MR144363}, we have 
$$
\forall k,N \in \mathbb{N}^+, \quad  
P\left(|T_N-\mathbb E T_N|>\varepsilon_k N\right)
    \le 2e^{-2\varepsilon_k^2N}.
$$
With the same reasoning above, since the right-hand side is summable in $N$, it follows by the first Borel-Cantelli lemma \cite[Theorem 4.18]{MR4226142} that $T_N=\mathbb E T_N+o(N)=o(N)$ $P$-almost surely. 
Hence there exists an $\mathscr{F}$-measurable subset $\Omega_1\subseteq \Omega$ such that $P(\Omega_1)=1$ and 
\begin{equation}\label{eq:finalclaim2}
\forall \omega \in \Omega_1, \quad 
\lim_{N\to \infty}\frac{|A(\omega)\cap [1,N]|}{N}
=\lim_{N\to \infty}\frac{T_N(\omega)}{N}
=\lim_{N\to \infty}\frac{\mathbb{E}T_N}{N}=0.
\end{equation}

To conclude the proof of Theorem \ref{thm:chucorrectedtechnical}, 
since $\Omega_0\cap \Omega_1$ is nonempty (as it is the intersection of two sets of full $P$-measure), 
it follows by \eqref{eq:finalclaim1} and \eqref{eq:finalclaim2} that there exists $A\subseteq \mathbb{N}^+$ such that $\mathsf{d}(A)=0$ and $\mathsf{d}(A+B)=\alpha$.


\section{Proof of Theorem \ref{thm:exponential}}

By hypothesis, it is possible to pick $q \in \mathbb{N}^+$ such that  
$$
c:=\frac{1}{2}\left(1+\liminf_{k\to \infty}\frac{b_{k+q}}{b_k}\right)>1.
$$
Hence, there exists $k_0 \in \mathbb{N}^+$ such that $b_{k+q}\ge cb_k$ for all $k\ge k_0$. 
As it follows by Theorem \ref{thm:chucorrectedtechnical}, it will be enough to show that there exists a function $r: \mathbb{N}^+ \to \mathbb{N}^+$ such that items \ref{item:0}--\ref{item:2} hold. 
To this aim, define 
$$
\forall n\in \mathbb N^+, \quad 
r(n):=\lfloor \sqrt n\rfloor.
$$
Clearly, $r(n)<n$ for all $n\ge 2$, and $r(n)=o(n)$ as $n\to\infty$, which shows items \ref{item:0}--\ref{item:1}. 

Now, pick $N,k\in \mathbb N^+$ such that $b_k\le N<b_{k+1}$. By repeated application of $b_{j+q}\ge cb_j$, for all sufficiently large $k$ we have
$$
b_k \ge c^{k/(3q)}b_{\lfloor k/2\rfloor}.
$$
Hence
$$
\sum_{2\le t\le k/2}b_{t-r(t)}
\le kb_{\lfloor k/2\rfloor}
\le \frac{k}{c^{k/(3q)}}\,b_k=o(b_k).
$$
Similarly, the sequence $(t-r(t))_{t\ge2}$ is nondecreasing. Therefore, for all sufficiently large $k$,
$$
\sum_{k/2<t\le k}b_{t-r(t)}
\le kb_{k-r(k)}
\le \frac{k}{c^{\lfloor r(k)/q\rfloor}}\,b_k
=o(b_k).
$$
Putting together the above estimates and recalling that $b_k\le N$, we obtain
\begin{equation}\label{eq:upperboundEr}
|E_r\cap [1,N]|
\le \sum_{2\le t\le k}b_{t-r(t)}
=o(N)
\qquad\text{as }N\to \infty.
\end{equation}
Therefore $\mathsf{d}(E_r)=0$ and item \ref{item:2} holds. This concludes the proof.


\section{Proof of Corollary \ref{cor:sumwithfinite}}

Set $q:=|C| \in \mathbb{N}^+$ and denote by $(x_n: n \in \mathbb{N}^+)$ the increasing enumeration of $B+C$. Taking into account that $\liminf_n b_{n+1}/b_n>1$ and that $B+C=\bigcup_{n}(\{b_n\}+C)$, it follows that $b_n+\max C<b_{n+1}+\min C$ for all large $n$, so that 
$$
\liminf_{n\to \infty} \frac{x_{n+q}}{x_n}\ge \liminf_{n\to \infty}\frac{b_{n+1}+\min C}{b_n+\min C}>1.
$$
The conclusion follows by Theorem \ref{thm:exponential}.


\section{Proof of Corollary \ref{cor:mixedfiniteunions}}

Discarding the finite sets among the $B_i$ does not affect the argument. Pick 
$c>1$ such that, for every infinite $B_i$, $b_{i,n+1}\ge c b_{i,n}$ for all sufficiently large $n$. Hence, for all sufficiently large $x$, each
$B_i$ contributes at most one element to the interval $[x,cx)$. Therefore $B$
contributes at most $q$ elements to $[x,cx)$. 
At this point, let $(b_n:n\in\mathbb N^+)$ be the increasing enumeration of $B$. By the above observation, we have $b_{n+q}\ge c b_n$ for all sufficiently large $n$. The conclusion follows by Theorem \ref{thm:exponential}.


\section{Proof of Theorem \ref{thm:lastmainthm}}

Set $a:=(\lambda-1)/2$ and define the map $r: \mathbb{N}^+\to \mathbb{N}^+$ by $r(1):=1$, $r(2):=1$, and 
$$
\forall n\ge 3, \qquad 
r(n):=\left\lfloor\frac{n}{(\log n)^a}\right\rfloor. 
$$
In particular, items \ref{item:0} and \ref{item:1} hold. 

Now, observe that, if $2\le k\le n/2$, then $b_{k-r(k)}\le b_{\lfloor n/2\rfloor}$. 
Furthermore,
$$
    \log\frac{b_n}{b_{\lfloor n/2\rfloor}}
    \gg \sum_{\lfloor n/2\rfloor\le j<n}
        \frac{(\log j)^\lambda}{j}
    \gg (\log n)^\lambda
    \quad \text{ as }n\to \infty. 
$$
Since $\lambda>1$, this implies $
    n b_{\lfloor n/2\rfloor}=o(b_n)$ as $n\to \infty$. 

Also, if $n/2<k\le n$, then $r(k)\gg n/(\log n)^a$, 
and hence $
    b_{k-r(k)}
    \le b_{n-c n/(\log n)^a}
$ 
for some constant $c>0$ and all large $n$. Therefore
$$
    \log\frac{b_n}{b_{n-c n/(\log n)^a}}
    \gg \sum_{n-c n/(\log n)^a\le j<n}
        \frac{(\log j)^\lambda}{j}
    \gg (\log n)^{\lambda-a}.
$$
Since $\lambda-a>1$, it follows that $n b_{n-c n/(\log n)^a}=o(b_n)$ as $n\to \infty$. 

Putting together the above estimates, we obtain 
$$
\sum_{2\le k\le n} b_{k-r(k)}=o(b_n)
\quad \text{ as }n\to \infty. 
$$
Reasoning as in \eqref{eq:upperboundEr}, this implies that $\mathsf{d}(E_r)=0$, hence item \ref{item:2} holds. 
The conclusion follows by Theorem \ref{thm:chucorrectedtechnical}.


\section{Proof of Theorem \ref{thm:logarithmic}}

For the cases $\alpha=0$ and $\alpha=1$ choose $A=\{1\}$ and $A=\mathbb{N}^+$, respectively. 
Hence, let us suppose hereafter that $\alpha\in (0,1)$. 

We start with the following elementary lemma.

\begin{lem}\label{lem:periodic}
Let $F\subseteq\mathbb{N}$ be a nonempty finite set, set $k:=|F|$, and fix $\alpha,\eta\in(0,1)$. Also, let $Q\ge 8/\eta$ be a power of $2$. 
Then there exist an integer $q \in [Q,4Q^{k+1}]$ and a $q$-periodic set $P\subseteq\mathbb{Z}$ such that
$$
\left|
\frac{|P\cap[0,q-1]|}{q}-\alpha
\right|
\le\frac1q
\quad\text{and}\quad
\frac{|((P+F)\setminus P)\cap[0,q-1]|}{q}\le\eta.
$$
\end{lem}

\begin{proof}
Write $F=\{f_1,\ldots,f_k\}$ and consider the points
$$
\left\{\left(\frac{rf_1}{4Q^{k+1}},\ldots,\frac{rf_k}{4Q^{k+1}}\right) \in \mathbb{T}^k:\,\, 0\le r\le2Q^{k+1}\right\},
$$
where $\mathbb{T}^k$ stands for the $k$-dimensional torus. Partitioning each coordinate into $Q$ intervals of length $1/Q$, one of the resulting $Q^k$ boxes contains at least $2Q+1$ of these points.

We claim that there exist two distinct indices $r,s \in [0,2Q^{k+1}]$ such that
$$
\frac{4Q^{k+1}}{\mathrm{gcd}(|r-s|,4Q^{k+1})}>Q.
$$
Indeed, otherwise, since $4Q^{k+1}$ and $Q$ are powers of $2$, the difference of every two indices in the same box would be divisible by $4Q^{k+1}/Q=4Q^k$. Hence all these indices would belong to the same residue class modulo $4Q^{k}$, which contains at most $Q/2+1$ integers in $[0,2Q^{k+1}]$, a contradiction.

Fix such indices $r,s$ and define $t:=|r-s|$, $g:=\mathrm{gcd}(t,4Q^{k+1})$, $q:=4Q^{k+1}/g$, and $u:=t/g$. Then $\mathrm{gcd}(u,q)=1$, $Q<q\le 4Q^{k+1}$, and, since the two corresponding points belong to the same box, we get 
$$
\left\|\frac{uf_i}{q}\right\|\le\frac1Q
\qquad\text{for all }i=1,\ldots,k.
$$

Define $J:=\{0,\ldots,\lfloor\alpha q\rfloor-1\}\subseteq\mathbb{Z}/q\mathbb{Z}$ and 
$$
P:=\{n\in\mathbb{Z}:un\bmod q\in J\}.
$$
It follows that $|P\cap[0,q-1]|=\lfloor\alpha q\rfloor$, which proves the first conclusion.

For each $i=1,\ldots,k$, let $a_i$ be the representative of $uf_i$ modulo $q$ such that $|a_i|\le q/Q$. Under the bijection $n\mapsto un\bmod q$, the set $P+f_i$ corresponds to $J+a_i$. Hence all the residue classes in $(P+F)\setminus P$ are contained in the two intervals adjacent to $J$, each of length at most $\lceil q/Q\rceil$. Therefore
$$
\frac{|((P+F)\setminus P)\cap[0,q-1]|}{q}
\le \frac{2}{q}\left\lceil \frac{q}{Q}\right\rceil
\le\frac2Q+\frac2q
<\frac4Q
\le\eta.
$$

This concludes the proof.
\end{proof}

We will repeatedly use the following consequence of Lemma \ref{lem:periodic}. If $I\subseteq\mathbb{Z}$ is a finite interval, then
\begin{equation}\label{eq:periodic-interval}
\left||P\cap I|-\alpha|I|\right|
\le2q+\frac{|I|}{q}
\quad\text{and}\quad
|((P+F)\setminus P)\cap I|
\le\eta|I|+2q.
\end{equation}
Indeed, decompose $I$ into pairwise disjoint intervals of length $q$, together with at most two remaining intervals whose total length is less than $2q$. By the $q$-periodicity of $P$, on each interval of length $q$ we have
$$
\left||P\cap J|-\alpha q\right|\le1
\quad\text{and}\quad
|((P+F)\setminus P)\cap J|\le\eta q.
$$
Since there are at most $|I|/q$ such intervals, while the two remaining intervals contribute at most $2q$, the estimates in \eqref{eq:periodic-interval} follow.

\medskip

Set $b:=\min B$ and $D:=B-b\subseteq\mathbb{N}$. In particular, $0\in D$. Since translation by an integer does not affect the hypothesis on $B$, we can fix $c\in (0,1)$ such that
\begin{equation}\label{eq:logarithmic-count}
|D\cap[0,n]|
\le c\,\frac{\log n}{\log\log n}
\end{equation}
for all sufficiently large $n$. Pick also a real number $\lambda>1/(1-c)$, so that
$$
\frac{\lambda-1}{c\lambda}>1.
$$

\medskip

We construct recursively two sequences $(\ell_j:j\in\mathbb{N}^+)$ and $(r_j:j\in\mathbb{N}^+)$ of elements of $D$ such that
$$
\ell_j\le r_j\le \min\{\ell_j^\lambda, \ell_{j+1}-1\}
\qquad\text{ and }\qquad
D\cap(r_j,\ell_{j+1})=\emptyset.
$$
Start with an arbitrary sufficiently large $\ell_1\in D$. Once $\ell_j$ has been chosen, list the elements of $D\cap[\ell_j,\ell_j^\lambda]$ as 
$
x_0=\ell_j<x_1<\cdots<x_m\le\ell_j^\lambda,
$ 
and let $x_{m+1}$ be the first element of $D$ larger than $\ell_j^\lambda$. Since $x_0=\ell_j$ and $x_{m+1}>\ell_j^\lambda$, we have
$$
\sum_{i=0}^{m}\log\left(\frac{x_{i+1}}{x_i}\right)
=
\log\left(\frac{x_{m+1}}{\ell_j}\right)
>
(\lambda-1)\log(\ell_j).
$$
Since $m+1\le |D\cap[0,\ell_j^\lambda]|$, it follows that, for some $i\in\{0,\ldots,m\}$,
$$
\log\left(\frac{x_{i+1}}{x_i}\right)
\ge
\frac{(\lambda-1)\log(\ell_j)}
{|D\cap[0,\ell_j^\lambda]|}.
$$
Define $r_j:=x_i$ and $\ell_{j+1}:=x_{i+1}$. By \eqref{eq:logarithmic-count},
\begin{equation}\label{eq:logarithmic-gap}
\log\frac{\ell_{j+1}}{r_j}
\ge
\left(\frac{\lambda-1}{c\lambda}-o(1)\right)
\log\log\ell_j.
\end{equation}

Now, we claim that
\begin{equation}\label{eq:gap-strong}
\log\frac{\ell_{j+1}}{r_j}-\log\log\ell_{j+1}\longrightarrow\infty
\qquad \text{as }   j\to \infty.
\end{equation}
Indeed, choose $\gamma>1$ so that the right-hand side of \eqref{eq:logarithmic-gap} is at least $\gamma\log\log\ell_j$ for all sufficiently large $j$. If $\log\ell_{j+1}\le2\lambda\log\ell_j$, then $\log\log\ell_{j+1}=\log\log\ell_j+O(1)$, and the claim follows from \eqref{eq:logarithmic-gap}. Otherwise, since $r_j\le\ell_j^\lambda$, we obtain
$$
\log\frac{\ell_{j+1}}{r_j}
\ge
\log\ell_{j+1}-\lambda\log\ell_j
>\frac12\log\ell_{j+1},
$$
and \eqref{eq:gap-strong} follows again.

For each $j\in\mathbb{N}^+$, set $D_j:=D\cap[0,r_j]$ and $k_j:=|D_j|$. Since $r_j\le\ell_j^\lambda$, it follows by \eqref{eq:logarithmic-count} that
\begin{equation}\label{eq:kj-bound}
k_j=O\left(\frac{\log\ell_j}{\log\log\ell_j}\right)
\qquad \text{as }   j\to \infty.
\end{equation}
In addition, by \eqref{eq:gap-strong},
\begin{equation}\label{eq:previous-shifts}
\frac{k_jr_{j-1}}{\ell_j}\longrightarrow0
\qquad \text{as }   j\to \infty.
\end{equation}
Indeed, $\log(r_{j-1}/\ell_j)=-\log(\ell_j/r_{j-1})$, whereas \eqref{eq:kj-bound} gives $\log k_j\le\log\log\ell_j+O(1)$.

\medskip

Set $a_j:=\log\ell_j$ and

$$
s_j:=\min\left\{j,\sqrt{\frac{a_j}{k_j}}\right\}.
$$
By \eqref{eq:kj-bound}, we have $s_j\to\infty$. Define $\eta_j:=e^{-s_j}$ and apply Lemma \ref{lem:periodic} to $D_j$, $\alpha$, and $\eta_j$. Denote by $P_j$ the resulting periodic set and by $q_j$ its period. Letting $Q_j$ be a power of $2$ with $8/\eta_j\le Q_j<16/\eta_j$, we obtain 
$$
\log q_j
\le O(1)+(k_j+1)(s_j+O(1))
=o(\log\ell_j)
\qquad \text{as }   j\to \infty.
$$
Indeed, $k_js_j\le\sqrt{k_ja_j}=o(a_j)$ by \eqref{eq:kj-bound}. Consequently
\begin{equation}\label{eq:qj-small}
q_j=\ell_j^{o(1)}
\quad\text{and}\quad
\frac{k_jq_j}{\ell_j}\longrightarrow0
\qquad \text{as }   j\to \infty.
\end{equation}

\medskip

Define $m_j:=\max\{r_{j-1},q_j,1\}$, where $r_0:=1$. By \eqref{eq:previous-shifts} and \eqref{eq:qj-small}, 
$
\lim_j k_jm_j/\ell_j=0.
$ 
For all sufficiently large $j$, set
$$
t_j:=\left\lceil\sqrt{\frac{m_j\ell_j}{k_j}}\right\rceil.
$$ 
It follows that
\begin{equation}\label{eq:tj-properties}
\frac{m_j}{t_j}\longrightarrow0 
\quad\text{and}\quad
\frac{k_jt_j}{\ell_j}\longrightarrow0
\qquad \text{as }   j\to \infty.
\end{equation}
In particular, $r_{j-1}<t_j<\ell_j$ for all sufficiently large $j$. Discarding finitely many initial terms, we can suppose that these inequalities hold for every $j$. Since 
$
t_{j-1}<\ell_{j-1}\le r_{j-1}<t_j,
$ 
the sequence $(t_j)$ is strictly increasing. Moreover, \eqref{eq:tj-properties} implies
\begin{equation}\label{eq:tj-small}
\frac{q_j}{t_j}\longrightarrow0,
\quad
\frac{r_{j-1}}{t_j}\longrightarrow0,
\quad
\frac{t_j}{\ell_j}\longrightarrow0,
\quad\text{and}\quad
\frac{t_j}{t_{j+1}}\longrightarrow0
\quad \text{as }   j\to \infty.
\end{equation}
For the last limit, it is enough to observe that $t_j<\ell_j\le r_j$ and apply the second limit in \eqref{eq:tj-small} with $j+1$ in place of $j$.

\medskip

We can now construct the required set. Define $C\subseteq\mathbb{N}^+$ so that
$$
C\cap[t_j,t_{j+1})
=
P_j\cap[t_j,t_{j+1})
$$
for all $j\in\mathbb{N}^+$, and $C\cap[0,t_{1}):=\emptyset$.

\medskip

We first show that $\mathsf{d}(C)=\alpha$. By \eqref{eq:periodic-interval}, applied to $[t_{j-1},t_j)$, we obtain

$$
\left|
|C\cap[1,t_j]|-\alpha t_j
\right|
\le
t_{j-1}+2q_{j-1}+\frac{t_j}{q_{j-1}}+O(1)
\quad \text{as }   j\to \infty.
$$
Taking into account \eqref{eq:tj-small} and the fact that $q_j\to\infty$, it follows that
\begin{equation}\label{eq:C-boundary}
\frac{|C\cap[1,t_j]|}{t_j}\longrightarrow\alpha
\quad \text{as }   j\to \infty.
\end{equation}
Now, if $t_j\le n<t_{j+1}$, another application of \eqref{eq:periodic-interval} gives
$$
\left|
|C\cap[1,n]|-\alpha n
\right|
\le
\left|
|C\cap[1,t_j]|-\alpha t_j
\right|
+2q_j+\frac{n-t_j}{q_j}+O(1)
\quad \text{as }   j\to \infty.
$$
Dividing by $n$ and using \eqref{eq:C-boundary} and \eqref{eq:tj-small}, we conclude that 
$
\mathsf{d}(C)=\alpha.
$

\medskip

Lastly, we prove that
\begin{equation}\label{eq:enlargement-zero}
\mathsf{d}((C+D)\setminus C)=0.
\end{equation}
Set $E:=(C+D)\setminus C$, fix $j\in\mathbb{N}^+$, and let $t_j\le n\le t_{j+1}$. Pick $z\in E\cap[t_j,n)$ and write $z=x+d$, with $x\in C$ and $d\in D$.

Suppose first that $x\ge t_j$. Since $z<t_{j+1}<\ell_{j+1}$ and $D\cap(r_j,\ell_{j+1})=\emptyset$, we have $d\le r_j$, hence $d\in D_j$. Therefore by \eqref{eq:periodic-interval} the number of such points $z$ is at most
$$
\eta_j(n-t_j)+2q_j.
$$
Suppose now that $x<t_j$. If $d\le r_{j-1}$, then $z<t_j+r_{j-1}$, so the number of such points is at most $r_{j-1}+O(1)$. If $d>r_{j-1}$, then, since $D\cap(r_{j-1},\ell_j)=\emptyset$, we have $d\ge\ell_j$. Hence no such representation exists if $n\le\ell_j$. On the other hand, if $n>\ell_j$, then $d<z<\ell_{j+1}$, so the gap $D\cap(r_j,\ell_{j+1})=\emptyset$ implies that $d\le r_j$. Thus $d\in D_j\setminus D_{j-1}$, and the number of pairs $(x,d)$ of this type is at most $k_jt_j$. We conclude that
\begin{equation}\label{eq:E-estimate}
|E\cap[t_j,n)|
\le
\eta_j(n-t_j)+2q_j+r_{j-1}
+\mathbf{1}_{\{n>\ell_j\}}k_jt_j+O(1).
\end{equation}
Applying \eqref{eq:E-estimate} to the preceding block with $n=t_j$, and using trivially $|E\cap[1,t_{j-1}]|\le t_{j-1}$, we obtain 
$$
\frac{|E\cap[1,t_j]|}{t_j}
\le
\frac{t_{j-1}}{t_j}
+\eta_{j-1}
+\frac{2q_{j-1}}{t_j}
+\frac{r_{j-2}}{t_j}
+\frac{k_{j-1}t_{j-1}}{t_j}
+o(1).
$$
Every term on the right-hand side tends to $0$. Indeed, this follows from \eqref{eq:tj-properties}--\eqref{eq:tj-small}, together with 
$
k_{j-1}t_{j-1}/t_j
\le
k_{j-1}t_{j-1}/\ell_{j-1}
\to 0
$ as $j\to \infty$. 
Hence
\begin{equation}\label{eq:E-boundary}
\frac{|E\cap[1,t_j]|}{t_j}\longrightarrow0
\quad \text{as }   j\to \infty.
\end{equation}
For an arbitrary $n$ with $t_j\le n<t_{j+1}$, it follows by \eqref{eq:E-estimate} and \eqref{eq:E-boundary} that
$$
\frac{|E\cap[1,n]|}{n}
\le
\frac{|E\cap[1,t_j]|}{t_j}
+\eta_j
+\frac{2q_j}{t_j}
+\frac{r_{j-1}}{t_j}
+\mathbf{1}_{\{n>\ell_j\}}\frac{k_jt_j}{n}
+o(1).
$$
If $n>\ell_j$, then $k_jt_j/n\le k_jt_j/\ell_j\to0$ as $j\to \infty$   by \eqref{eq:tj-properties}; all the remaining terms tend to $0$ by \eqref{eq:tj-small}. This proves \eqref{eq:enlargement-zero}.

\medskip

Since $0\in D$, we have $C\subseteq C+D$. Hence $\mathsf{d}(C)=\alpha$ and \eqref{eq:enlargement-zero} yield
$$
\mathsf{d}(C+D)=\alpha.
$$
To conclude, recall that $D=B-b$, and define
$$
A:=\{x-b:x\in C,\ x>b\}\subseteq\mathbb{N}^+.
$$
Then
$$
A+B=(C\cap(b,\infty))+D\subseteq C+D,
$$
and
$$
(C+D)\setminus(A+B)
\subseteq
(C\cap[1,b])+D.
$$
The set on the right is a finite union of translates of $D$. Since \eqref{eq:logarithmic-count} implies that $\mathsf{d}(D)=0$, it follows that 
$
\mathsf{d}\bigl((C+D)\setminus(A+B)\bigr)=0.
$ 
Therefore $\mathsf{d}(A)=\mathsf{d}(C)=\alpha$ and $\mathsf{d}(A+B)=\mathsf{d}(C+D)=\alpha$, as desired.

\bigskip

\section{Proof of Theorem \ref{thm:counterexample}}

We start with the following probabilistic lemma.

\begin{lem}\label{lem:counterexample-random-block}
Fix integers $h,\ell,m,k\in\mathbb{N}^+$ and reals $p,\eta\in(0,1]$. Suppose that
$$
\frac{1}{p}\log\left(\frac{4}{\eta}\right)\le k\le m
\quad\text{and}\quad
\ell\ge \frac{10mk^2}{\eta^2}.
$$

Then there exists a set $D\subseteq [h+1,h+\ell]$ such that the following hold\textup{:}
\begin{enumerate}[label={\rm (\roman{*})}]
\item \label{item:counterexample-covering}
for every $F\subseteq [1,m]$ with $|F|\ge k$, we have

$$
|(F+D)\cap [h+m+1,h+\ell]|
\ge \ell-m-\eta\ell\textup{;}
$$

\item \label{item:counterexample-prefix}
for every $t\in\{1,\ldots,\ell\}$, we have

$$
|D\cap [h+1,h+t]|
\le 2pt+\frac{1}{p}\log(8\ell)\textup{.}
$$

\end{enumerate}
\end{lem}

\begin{proof}
Pick a probability measure space $(\Omega, \mathscr{F}, P)$ and independent Bernoulli random variables $(X_n: n \in [h+1,h+\ell])$ such that $P(X_n=1)=p$ for all $n \in [h+1,h+\ell]$. Then, for each $\omega \in \Omega$, define 
$$
D(\omega):=\{n \in [h+1,h+\ell]: X_n(\omega)=1\}.
$$
Let $\Omega_1$ and $\Omega_2$ be the measurable events where items \ref{item:counterexample-covering} and \ref{item:counterexample-prefix} fail, respectively. 


\medskip

We first prove that $P(\Omega_1)<1/8$. It is enough to consider sets $F\subseteq [1,m]$ with $|F|=k$. Fix one such set $F$, and for each $\omega\in \Omega$ define 
$$
U_F(\omega):=\left|[h+m+1,h+\ell]\setminus (F+D(\omega))\right|.
$$
For each $n\in [h+m+1,h+\ell]$, the integers $n-f$, with $f\in F$, are pairwise distinct and belong to $[h+1,h+\ell]$. Therefore
$$
P(n\notin F+D)=(1-p)^k\le e^{-pk}\le \frac{\eta}{4},
$$
which implies that $\mathbb{E}U_F\le \eta\ell/4$.

Changing the membership in $D$ of a single integer can modify $U_F$ by at most $k$. Hence, by McDiarmid's inequality \cite[Lemma 1.2]{MR1036755}, we obtain

$$
P(U_F>\eta\ell)
\le
2\,\mathrm{exp}\left(-2\,\frac{(3\eta\ell/4)^2}{\ell k^2}\right)
=
2\,\mathrm{exp}\left(-\frac{9\eta^2\ell}{8k^2}\right).
$$

There are at most $2^m$ choices for $F$. Hence the probability that item \ref{item:counterexample-covering} fails for at least one $F$ is at most

$$
2^{m+1}\,\mathrm{exp}\left(-\frac{9\eta^2\ell}{8k^2}\right)
\le
2\,\mathrm{exp}\left(\left(\log 2-\frac{90}{8}\right)m\right)
<\frac18.
$$

\medskip

We prove 
that $P(\Omega_2)<1/8$ 
similarly. Fix $t\in\{1,\ldots,\ell\}$ and define $S_t(\omega):=|D(\omega)\cap [h+1,h+t]|$, so that $\mathbb{E}S_t=pt$. Set also $r:=\log(8\ell)/p$. Since

$$
\frac{(pt+r)^2}{t}
=
p^2t+2pr+\frac{r^2}{t}
\ge 2pr+2\sqrt{p^2t\cdot \frac{r^2}{t}}
=4pr,
$$
it follows again by McDiarmid's inequality that
$$
P(S_t>2pt+r)
\le
2\,\mathrm{exp}\left(-2\,\frac{(pt+r)^2}{t}\right)
\le
2e^{-8pr}
=
2(8\ell)^{-8}.
$$

Therefore the probability that item \ref{item:counterexample-prefix} fails for at least one $t\in\{1,\ldots,\ell\}$ is at most $2\ell(8\ell)^{-8}<1/8$. 

Putting everything together, we get $P(\Omega\setminus (\Omega_1\cup \Omega_2))>1/2$, so that there exists $\omega \in \Omega$ for which $D(\omega)$ satisfies both items \ref{item:counterexample-covering}--\ref{item:counterexample-prefix}. This concludes the proof.
\end{proof}

At this point, we proceed with the construction of $B$. Fix $\varepsilon>0$, and define

$$
\delta:=\min\left\{\frac{\varepsilon}{2},\frac{1}{12}\right\}
\quad\text{and}\quad
\theta:=\frac23+\delta.
$$

In particular, $\nicefrac{2}{3}<\theta<1$ and $\theta<\nicefrac{2}{3}+\varepsilon$.

Set $B_3:=\{1\}$ and $\rho_3:=1$. Suppose that, for some $j\ge4$, a finite set $B_{j-1}\subseteq [1,\rho_{j-1}]$ has been constructed, and set $s_{j-1}:=|B_{j-1}|$. Pick an integer $m_j>\rho_{j-1}$ sufficiently large, and define
\begin{displaymath}
\begin{split}
k_j&:=\left\lfloor\frac{m_j}{j s_{j-1}}\right\rfloor,
\qquad
p_j:=\frac{2j s_{j-1}\log(4j)}{m_j},
\qquad
\eta_j:=\frac1j,\\
h_j&:=m_j^2,
\qquad
\ell_j:=\left\lceil\frac{10m_j^3}{s_{j-1}^2}\right\rceil,
\quad \text{ and }\quad
r_j:=\frac{1}{p_j}\log(8\ell_j).
\end{split}
\end{displaymath}
Observe that we can choose $m_j$ so large that
\begin{displaymath}
\begin{split}
\frac{m_j}{j s_{j-1}}&\ge2,\qquad
p_j\le1,\qquad
\frac{h_j}{\ell_j}\le\frac1j,\qquad
\frac{m_j}{\ell_j}\le\frac1j,\\
s_{j-1}&\le h_j^\theta,\qquad
r_j\le h_j^\theta,
\quad \text{ and }\quad
2p_j(h_j+\ell_j)^{1-\theta}\le1.
\end{split}
\end{displaymath}
Indeed, with $j$ and $s_{j-1}$ fixed and $m_j\to\infty$, we have 
$$
r_j=O_{j,s_{j-1}}(m_j\log m_j)
\quad \text{ and }\quad
p_j(h_j+\ell_j)^{1-\theta}
=
O_{j,s_{j-1}}(m_j^{2-3\theta})
=o(1),
$$
where the last equality follows from $\theta>\nicefrac{2}{3}$. The remaining conditions are immediate for all sufficiently large $m_j$.

By the choice of $m_j$, we have

$$
k_j\ge\frac{m_j}{2j s_{j-1}},
\quad
p_jk_j\ge\log(4j)=\log\left(\frac4{\eta_j}\right),
\quad \text{ and }\quad
\frac{10m_jk_j^2}{\eta_j^2}
\le\frac{10m_j^3}{s_{j-1}^2}\le\ell_j.
$$

Hence, by Lemma \ref{lem:counterexample-random-block}, there exists a set $D_j\subseteq [h_j+1,h_j+\ell_j]$ satisfying items \ref{item:counterexample-covering}--\ref{item:counterexample-prefix}. Define $B_j:=B_{j-1}\cup D_j$ and $\rho_j:=h_j+\ell_j$. Since $h_j=m_j^2>m_j>\rho_{j-1}$, we have $\min D_j>\max B_{j-1}$. Finally, define
$$
B:=\bigcup_{j\ge3}B_j
=
\{1\}\cup\bigcup_{j\ge4}D_j.
$$

\medskip

Notice that each $D_j$ is nonempty. Indeed, by the choice of $m_j$,
$$
\ell_j-m_j-\frac{\ell_j}{j}
\ge
\ell_j\left(1-\frac2j\right)>0,
$$
and hence item \ref{item:counterexample-covering} would fail if $D_j=\emptyset$. Therefore $B$ is infinite.

\medskip

At this point, we claim that
\begin{equation}\label{eq:counterexample-growth}
|B\cap[1,n]|\le3n^\theta
\end{equation}
for all sufficiently large $n\in\mathbb{N}^+$. To this aim, suppose first that $h_j\le n\le h_j+\ell_j$ for some $j\ge4$, and set $t:=n-h_j$. If $t\ge1$, then item \ref{item:counterexample-prefix} gives

$$
|B\cap[1,n]|
\le
s_{j-1}+2p_jt+r_j
\le 3n^\theta.
$$

Indeed, $s_{j-1}\le h_j^\theta\le n^\theta$ and $r_j\le h_j^\theta\le n^\theta$, while
$
2p_jt\le2p_jn\le n^\theta
$
by the choice of $m_j$. The case $t=0$ is immediate. 
If instead $\rho_j\le n<h_{j+1}$, then $B\cap(\rho_j,n]=\emptyset$. Therefore, using the estimate above at $\rho_j$, we obtain 
$
|B\cap[1,n]|
=
|B\cap[1,\rho_j]|
\le3\rho_j^\theta
\le3n^\theta.
$ 
This proves \eqref{eq:counterexample-growth}. Since $\theta<2/3+\varepsilon$, it follows that 
$$
|B\cap[1,n]|=O(n^{2/3+\varepsilon})
\quad\text{as }n\to\infty.
$$

\medskip

It remains to prove the last claim. Fix $A\subseteq\mathbb{N}^+$ such that $A+B$ admits asymptotic density, and set $\alpha:=\mathsf{d}(A+B)$. There is nothing to prove if $\alpha=0$, hence suppose hereafter that $\alpha>0$. Since $m_j\to\infty$ and $1/j\to0$ as $j\to \infty$, for all sufficiently large $j$ we have
$$
|(A+B)\cap[1,m_j]|\ge\frac{m_j}{j}.
$$ 
On the other hand, since $\rho_{j-1}<m_j< h_j$, we have $|B\cap[1,m_j]|=|B_{j-1}|=s_{j-1}$. Define $A_j:=A\cap[1,m_j]$. Every integer in $(A+B)\cap[1,m_j]$ is the sum of an element of $A_j$ and an element of $B\cap[1,m_j]$. Therefore
$$
\frac{m_j}{j}
\le
|(A+B)\cap[1,m_j]|
\le
|A_j|s_{j-1},
$$
which implies that $|A_j|\ge m_j/(j s_{j-1})\ge k_j$.

Applying item \ref{item:counterexample-covering} to $A_j$ and recalling that $D_j\subseteq B$, we obtain

$$
|(A+B)\cap[1,h_j+\ell_j]|
\ge
\ell_j-m_j-\frac{\ell_j}{j}.
$$

Hence

$$
\frac{|(A+B)\cap[1,h_j+\ell_j]|}{h_j+\ell_j}
\ge
\frac{1-m_j/\ell_j-1/j}{1+h_j/\ell_j}
\ge
\frac{1-2/j}{1+1/j}
\longrightarrow1
\quad 
\text{ as }j\to \infty.
$$ 
It follows that
$$
\limsup_{n\to\infty}
\frac{|(A+B)\cap[1,n]|}{n}=1.
$$ 
Therefore $\alpha=1$. We conclude that $\mathsf{d}(A+B)\in\{0,1\}$ whenever $A+B$ admits asymptotic density.


\bibliographystyle{amsplain}

\end{document}